\documentclass[12pt]{amsart}
\usepackage{a4wide}
\usepackage[utf8]{inputenc}
\usepackage{amsmath}
\usepackage{amsthm}
\usepackage{amssymb}
\usepackage{amsfonts}
\usepackage{amsopn}
\usepackage{graphicx}
\usepackage{enumerate}
\usepackage{color}
\usepackage{mathtools}
\usepackage{mathrsfs}
\usepackage{bbm}
\usepackage{yhmath}
\usepackage{cite}
\usepackage[colorlinks,linkcolor=blue,anchorcolor=blue,citecolor=blue]{hyperref}

\newtheorem{theorem}{Theorem}[section]
\newtheorem{proposition}[theorem]{Proposition}

\newtheorem{corollary}[theorem]{Corollary}

\theoremstyle{definition}
\newtheorem{example}[theorem]{Example}

\newtheorem*{conjecture}{Conjecture} 
\newtheorem*{question}{Question}

\numberwithin{equation}{section}

\newcommand{\mbb}{\mathbb}

\newcommand{\mbf}{\mathbf}

\newcommand{\set}[1]{\left\{ #1 \right\}}

\newcommand{\R}{\mathbb{R}}

\newcommand{\Z}{\mathbb{Z}}
\newcommand{\N}{\mathbb{N}}

\newcommand{\f}{\infty}

\newcommand{\sse}{\subseteq}

\newcommand{\D}{\;\mathrm{d}}

\newcommand{\red}[1]{{\color{red}#1}}

\title[Existence and spectrality]{Existence  and spectrality of infinite convolutions generated by infinitely many admissible pairs}

\author[J. J. Miao]{Jun Jie Miao}
\address[J. J. Miao]{School of Mathematical Sciences, Shanghai Key Laboratory of PMMP, East China Normal University, Shanghai 200241, People's Republic of China}
\email{jjmiao@math.ecnu.edu.cn}

\author[H. Zhao]{Hongbo Zhao}
\address[H. Zhao]{School of Mathematical Sciences, Shanghai Key Laboratory of PMMP, East China Normal University, Shanghai 200241,
	People's Republic of China}
\email{2504245357@qq.com}

\date{\today}

\begin{document}
	
\maketitle
\begin{abstract}
In this paper, we study the spectrality of infinite convolutions generated by infinitely many admissible pairs which may not be compactly  supported, where the spectrality means the corresponding square integrable function space admits a family of exponential functions as an orthonormal basis.

First, we introduce remainder bounded condition (RBC), and we show it is a sufficient condition for  the existence of  the infinite convolution. Then we prove that the infinite convolution generated by a sequence of admissible pairs is a  spectral measure if the sequence of admissible pairs  satisfies RBC, and it has a subsequence consisting of general consecutive sets. Next, we show that the subsequence of general consecutive sets may be replaced by a general assumption, named partial concentration condition (PCC). Finally, we investigate the infinite convolutions generated by special subsequences,  and we give sufficient conditions for the spectrality  of such infinite convolutions.
\end{abstract}

\section{Introduction}
\subsection{Spectral measures and iterated function systems}
A Borel probability measure $\mu$ on $\R^d$ is called a \emph{spectral measure} if there exists a countable subset $\Lambda \sse \R^d$ such that the set of exponential functions $\set{e_\lambda(x) = e^{-2\pi i \lambda \cdot x}: \lambda \in \Lambda}$ forms an orthonormal basis in $L^2(\mu)$, and the set $\Lambda$ is often called a \emph{spectrum} of $\mu$.
The existence of spectrum of measures is a fundamental question in harmonic analysis, which was first studied by Fuglede for the normalized Lebesgue measure on measurable sets in ~\cite{Fuglede-1974}.
In the paper, Fuglede proposed his famous conjecture.  We write $\mbb{L}$ for the $d$-dimensional Lebesgue measure on $\R^d$ and $\frac{1}{\mbb{L}(A)} \mbb{L}|_A$ for the normalized Lebesgue measure restricted on a subset  $A\subset\R^d$. Note that the measure $\frac{1}{\mbb{L}(A)} \mbb{L}|_A$ may still be regarded as a measure defined on $\R^d$ since the measure outside $A$ is zero. 
\begin{conjecture}
\emph{A measurable set $\Gamma \sse \R^d$ with $0<\mbb{L}(\Gamma)<+\infty$ is a spectral set, that is, the  measure $\frac{1}{\mbb{L}(\Gamma)} \mbb{L}|_\Gamma$ is a spectral measure if and only if $\Gamma$ tiles $\R^d$ by translations.}
\end{conjecture}
Fuglede's conjecture has been disproved for $d \ge 3$, but the conjecture remains open for $d=1$ and $d=2$, see \cite{Farkas-Matolcsi-Mora-2006,Farkas-Revesz-2006,Kolountzakis-Matolcsi-2006a,Kolountzakis-Matolcsi-2006b,Matolcsi-2005, Tao-2004} for details. However the connection between spectrality and tiling has raised huge interest, and some affirmative results have been proved for special cases \cite{Iosevich-Katz-Tao-2003,Laba-2001}.

Many measures in fractal geometry are singular to Lebesgue measures,  that is, these measures are concentrated on  sets of zero Lebesgue measure. There is a big difference in the theory of spectrum between absolutely continuous measures and singular measures. 
An iterated function system (IFS) is a finite family of contractions $S_1,\ldots,S_m: \R^d \rightarrow \R^d$ with $m\geq 2$, and there exists a non-empty compact subset $E\subset \R^d$ such that 
$$
E=\bigcup_{i=1}^m S_i (E),
$$
which is called \emph{the attractor of the IFS $\{S_i\}_{i=1}^m$}. If the IFS   $\{S_i\}_{i=1}^m$ consists of affine contractions $S_i$: $\R^d\rightarrow \R^d$
\begin{equation}
S_i(x)=T_i(x)+a_i,\quad \quad i=1,2,\ldots,m,  \label{IFS}
\end{equation}
where $a_i\in \R^d$ is a translation vector and $T_i$ is a
non-singular linear mapping, we call the attractor $E$ \emph{ a self-affine set}. Given a probability vector $(p_1,\ldots, p_m)$ and an IFS   $\{S_i\}_{i=1}^m$ consisting of affine contractions, there exists a unique measure $\mu $  satisfying that  
\begin{equation}\label{msa}
\mu(A)=\sum_{i=1}^m p_i \mu  (S_i^{-1}(A)),
\end{equation}
for all Borel subsets $A \subseteq \R^d$, we call  $\mu$  {\em
a self-affine measure}. Note that the support of $\mu$ is $E$.  If the IFS   $\{S_i\}_{i=1}^m$ consists of similarity mappings, the attractor $E$ is called \emph{a self-similar set}, and $\mu$ is called
{\em  a self-similar measure}. It is clear that  self-similar sets and measures are special cases of self-affine sets and measures.  We refer readers to \cite{Falco03} for the background reading of fractal geometry.

In 1998, Jorgensen and Pedersen~\cite{Jorgensen-Pedersen-1998}  found that some self-similar measures may have spectra. For example, let  $\{S_1(x)=\frac{1}{4}x,  S_2(x)=\frac{1}{4}x+\frac{1}{2}\}$ and $p_1=p_2=\frac{1}{2}$. Then the corresponding self-similar measure  $\mu$ given by \eqref{msa}, i.e.
$$
\mu(A) = \frac{1}{2} \mu(S_1^{-1}(A)) + \frac{1}{2} \mu(S_2^{-1}(A)),
$$
for all Borel $A\subset\R$, is a spectral measure with a spectrum $$\Lambda = \bigcup_{n=1}^\f \set{\ell_1 + 4\ell_2 + \cdots + 4^{n-1} \ell_n: \ell_1,\ell_2,\cdots,\ell_n \in \set{0,1}}.$$
From then on, the spectrality of various self-similar measures and self-affine measures has been extensively studied, see \cite{An-Fu-Lai-2019,An-He-He-2019,An-He-2014,An-Lai-2020, An-Wang-2021,Dai-He-Lau-2014,Deng-Chen-2021,
Dutkay-Haussermann-Lai-2019,Fu-Wen-2017,Laba-Wang-2002,Li-2009,Liu-Dong-Li-2017,
He-Tang-Wu-2019} for various studies.

Admissible pairs are the key to study the spectrality of self-affine measures. Let $N$ be a $d \times d$ expansive matrix ( all eigenvalues have modulus strictly greater than $1$) with integer entries. Let $B\sse \Z^d$ be a finite subset of integer vectors with $\# B\ge 2$, where  $\#$ denotes the cardinality of a set.  If there exists $L\sse \Z^d$ with $\# B=\# L$ such that  the matrix 
\begin{equation}\label{unimat}
\left[ \frac{1}{\sqrt{\# B}} e^{-2 \pi i  (N^{-1}b)\cdot\ell }  \right]_{b \in B, \ell \in L} 
\end{equation}
is unitary, we call $(N, B)$ an {\it admissible pair} in $\R^d$. To emphasize on the dependence of $L$, we call $(N, B, L)$  a {\it Hadamard triple}, see~\cite{Dutkay-Haussermann-Lai-2019} for details.

 Given an  admissible pair $(N, B)$, the set $\{S_b(x)=N^{-1}(x+b)\}_{b\in B} $ forms an  IFS consisting of affine mappings. Setting the probability vector $(p_b=\frac{1}{\# B})_{b\in B}$, we write  $\mu_{N,B}$ for the corresponding self-affine measure given by \eqref{msa}  to emphasize the dependence on   $(N,B)$. Note that if $d=1$, $\mu_{N,B}$ is a self-similar measure.  In \cite{Laba-Wang-2002}, {\L}aba and Wang proved that the self-similar measure   $\mu_{N,B}$  in $\R$ is  a spectral measure. In~\cite{Dai-He-Lau-2014}, Dai, He and Lau completely resolved the spectrality  of the self-similar measure $\mu_{N,B}$ given by $B=\{0,1,\ldots, \# B-1\}$ satisfying  $\# B | N$ ( this means that $\# B$ is a divisor of $N$). In 2019, Dutkay, Haussermann and Lai  generalized {\L}aba and Wang's work to higher dimensional case in \cite{Dutkay-Haussermann-Lai-2019}, and they proved that the self-affine measure $\mu_{N,B}$ in $\R^d$ is   spectral.

\subsection{Infinite convolusions} 
For a finite subset $A \sse \R^d$, the uniform discrete measure supported on $A$ is given by
\begin{equation*}
\delta_A = \frac{1}{\# A} \sum_{a \in A} \delta_a,
\end{equation*}
where $\delta_a$ denotes the Dirac measure at the point $a$. Given a sequence $\{(N_n,B_n)\}_{n=1}^\infty $, for each integer $n\geq 1$, we write
\begin{equation}\label{def_mun}
  \mu_n =\delta_{N_1^{-1} B_1} * \delta_{(N_2N_1)^{-1} B_2} * \cdots * \delta_{(N_n \cdots N_2 N_1)^{-1} B_n},
\end{equation}
where $*$ denotes the convolution between measures.
If the sequence $\{\mu_n\}_{n=1}^\infty$ converges weakly to a Borel probability measure, the weak limit is called the \emph{infinite convolution} of $\{(N_n,B_n)\}_{n=1}^\infty$, denoted by
\begin{equation}\label{infinite-convolution}
  \mu = \delta_{N_1^{-1} B_1} * \delta_{(N_2N_1)^{-1} B_2} * \cdots * \delta_{(N_n \cdots N_2 N_1)^{-1} B_n} * \cdots .
\end{equation}
We refer readers to \cite{Folla99} for details. 
Note that $\mu_n$ may not converge in the weak topology, see Example~\ref{not converge}.
If infinite convolution $\mu$ exists, we write
\[\mu_{>n}=\delta_{(N_{n+1} \cdots N_2 N_1)^{-1}B_{n+1}}\ast\delta_{(N_{n+2} \cdots N_2 N_1)^{-1}B_{n+2}}\ast\cdots,\]
and we rewrite  $\mu =\mu_n\ast\mu_{>n}$.
For $n\ge1$, we write
\begin{equation}\label{def_nugn}
 \nu_{>n}=\delta_{{N_{n+1}}^{-1}B_{n+1}}\ast\delta_{(N_{n+2}N_{n+1})^{-1}B_{n+2}}\ast\cdots, 
\end{equation}
which is equivalent to  $\nu_{>n}( \cdot) =\mu_{>n}((N_n \cdots N_2 N_1)^{-1}\cdot)$.  Note that if all the elements in $\{(N_n,B_n)\}_{n=1}^\infty $ are identical to $(N, B)$, the corresponding  infinite convolution degenerates to a self-affine measure $\mu_{N,B}$.

It is natural to investigate the spectrality of infinite convolutions since the self-affine measures $\mu_{N,B}$ generated by admissible pairs $(N, B)$ may be represented by infinite convolutions.
The spectrality of the infinite convolution generated by admissible pairs was first studied by Strichartz~\cite{Strichartz-2000}.
From then on, huge interest has been aroused by the spectral question on infinite convolutions. Since there exist infinite convolutions generated by admissible pairs which are not spectral measures, see Example~\ref{ex_NE}, it is natural to raise the following question.
\begin{question}
\emph{ Given a sequence of admissible pairs $\{(N_k, B_k)\}_{k=1}^\f$, under what assumptions is the infinite convolution
\begin{equation} \label{def_admu}
 \mu = \delta_{N_1^{-1} B_1} * \delta_{(N_2N_1)^{-1} B_2} * \cdots * \delta_{(N_n \cdots N_2 N_1)^{-1} B_n} * \cdots 
\end{equation} 
 }
spectral?
\end{question}
Note that the support of $\mu$ given by \eqref{def_admu} is
 $$T_\mu=\bigg\{\sum_{n=1}^{\f}(N_n\cdots N_2N_1)^{-1}b_n:b_n\in B_n \ \text{for all}\ n\ge 1\bigg\}.$$
For every $\mbf b\in N_1^{-1}B_1+(N_2N_1)^{-1}B_2+\cdots+(N_n\cdots N_2N_1)^{-1}B_n$, we write 
 \[T_{\mbf b}=\mbf b+\bigg\{\sum_{j=1}^{\f}(N_{n+j}\cdots N_{n+2}N_{n+1})^{-1}b_j:b_j\in B_j \ \text{for all}\ j>n\bigg\}.\]
We say that $\mu$ satisfies {\it no-overlap condition} if $\mu(T_{\mbf b}\cap T_{\mbf b'})=0$	for all $\mbf b\neq \mbf b'\in N_1^{-1}B_1+(N_2N_1)^{-1}B_2+\cdots+(N_n\cdots N_2N_1)^{-1}B_n$ for all $n\in\N$.  In 2017, Dutkey and Lai \cite{Dutkay-Lai-2017} proved that the infinite convolution $\mu$ is a spectral measure if $\mu$ is compactly supported and satisfies  no-overlap condition,  and if the infimum of the singular values of $\frac{1}{\#  B_n}[|\hat{\mu}_{>n}(l)|e^{-2\pi ib\cdot l}]_{l\in L_n, b\in B_n}$ is positive where $\hat{\mu}$ is the Fourier transform of $\mu$.
In 2019, An, Fu and Lai  \cite{An-Fu-Lai-2019} studied this question in $\R$ and proved that the infinite convolution $\mu$ is spectral if $B_n\subset\{0,1,\ldots,N_n-1\}$ for all $n\geq 1$ and $\liminf_{n\to \infty} \#  B_n <\infty$. Recently, Li, Miao and Wang~\cite{LMW-2023} studied the spectrality of infinite convolutions $\mu$ generated by finitely many admissible pairs. We refer the readers to ~\cite{An-Fu-Lai-2019,An-He-Lau-2015,An-He-Li-2015,Dai-2012, Dutkay-Haussermann-Lai-2019,
Dutkay-Lai-2017,Fu-Wen-2017,Laba-Wang-2002} for further studies on infinite convolutions.

\section{Main results}\label{sec_mr}
Currently, the spectral theory of infinite convolutions mainly focuses on the measures generated by finitely many admissible pairs or with compact support. In this paper, we study the infinite convolution  generated by infinitely many admissible pairs  on $\R$, and such measures  may not be compactly  supported, see Example~\ref{ex_ncpt}.

From now on,   for each given sequence $\{( N_n,B_n)\}_{n=1}^\infty $, we always assume that $N_n>0$ and $B_n\subseteq\R$ with $\#B_n\geq 2$ for all $n\ge1$.

To obtain the weak limit of $\{\mu_n\}_{n=1}^\f$ generated by a sequence  $\{( N_n,B_n)\}_{n=1}^\infty $,  it is essential to make assumptions on the sequence. Given a sequence $\{( N_n,B_n)\}_{n=1}^\infty $, we say that $\{( N_n,B_n)\}_{n=1}^\infty $ satisfies {\it remainder bounded condition (RBC)} if
	\[\sum_{n=1}^{\infty} \frac{\# B_{n,2}}{\# B_n} < \infty,\]
where $B_{n,1}=B_n \cap \{0,1,\cdots,N_n-1\}$ and $B_{n,2}=B_n \backslash B_{n,1}.$
Note that if $B_n\subset \{0,1,\ldots, N_n-1\}$ for all $n\ge1$, it is clear that $\{( N_n,B_n)\}_{n=1}^\infty $ satisfies   RBC, moreover, the corresponding infinite convolution  exists and has compact support.
Since we do not assume $B_n\subset \{0,1,\ldots, N_n-1\}$, we come up with RBC  in this paper, and it plays a fundamental role in the study of infinite convolutions generated by infinitely many admissible pairs.

Our first conclusion shows that RBC guarantees the existence of infinite convolutions.
\begin{theorem}\label{existence}
Given $\{( N_n,B_n)\}_{n=1}^\infty $. If  the sequence satisfies either  RBC or 
$$
\sum_{n=1}^{\f}\frac{\max\{|b|:b\in B_n\}}{N_1N_2\cdots N_n}<\f.
$$ 
Then the infinite convolution $\mu$ given by \eqref{infinite-convolution} exists. \

If there exists $r_0>0$ such that 
$$
\sum_{n=1}^{\infty} \frac{\# \{b:\ b\in B_n,\ |b|>N_1N_2\cdots N_nr_0\}}{\# B_n} = \infty,
$$  
the infinite convolution  $\mu$ does not exist.
\end{theorem}

Next, we investigate the spectrality of  infinite convolutions generated by  admissible pairs containing some special subsequences.
Given an admissible pair $( N,B) $, we call $B$ a {\it general consecutive set}  if
	\[B\equiv \{0,1,2,\cdots, \# B-1\} \pmod{N}.\]
Note that $ \# B | N$,  that is, $\# B$ is a divisor of $N$.  Since $( N,B) $ is an admissible pair, there exists  $L\subseteq\Z$ such that $(N, B, L)$ is a Hadamard triple. Then  $( N,\{0,1,\cdots, \# B-1\}, L) $ is also  a Hadamard triple. If $\# B \nmid N$, then the equation $\sum_{b=0}^{\# B-1}e^{2\pi i (N^{-1}b)\cdot x}=0$ has no integer solution, which contradicts the fact that  $L\subseteq\Z$ and the matrix \eqref{unimat} is unitary.

Our second main conclusion shows that  infinite convolutions are spectral measures if  general consecutive sets appear infinitely many times in the sequence of admissible pairs.
\begin{theorem}\label{infinitely many admissible pairs}
Let $\{( N_n,B_n)\}_{n=1}^\infty $ be a sequence of admissible pairs satisfying {\it RBC}. Suppose that there exists  a subsequence $\{B_{n_k}\}_{k=1}^\infty$ of general consecutive sets. Then the infinite convolution $\mu$  given by \eqref{infinite-convolution} exists and is a spectral measure.
\end{theorem}

Immediately, we have the following conclusion for general consecutive sets.
\begin{corollary}
Let $\{(N_n,B_n)\}_{n=1}^\infty$ be  a sequence  of admissible pairs of general consecutive sets with {\it RBC} satisfied. Then the infinite convolution $\mu$  given by \eqref{infinite-convolution} exists and  is a spectral measure.
\end{corollary}
Note that the infinite convolution $\mu$ may not exist even if $\mu$ is generated by a sequence of admissible pairs only consisting of  general consecutive sets, see Example~\ref{not converge}. There also exist examples that the infinite convolution generated by general consecutive sets is not a spectral measure, see Example~\ref{ex_NE}. Therefore the remainder bounded condition (RBC) is essential in our context.

Suppose that $B_n = \{0,1,\dots, \# B_n-1\}$ for all integers $n\geq1$. It is clear that $B_{n,2}$ is empty, and  the sequence $\{( N_n,B_n)\}_{n=1}^\infty $  satisfies   RBC. The following conclusion is an immediate consequence of Theorem \ref{infinitely many admissible pairs}, and it shows that the infinite convolutions generated by such special general consecutive sets are spectral measures.
\begin{corollary}
Let $\{(N_n,B_n)\}_{n=1}^\infty$ be  a sequence  of admissible pairs with $B_n = \{0,1,\dots, \# B_n-1\}$. Then the infinite convolution $\mu$  given by \eqref{infinite-convolution} exists and  is a spectral measure.
\end{corollary}

Since general consecutive  sets are special, it is better to replace it by some more general assumptions.  Given $l\in (0,1)$, we say $\{( N_n,B_n)\}_{n=1}^\infty $ satisfies  \textit{partial concentration condition (PCC)}, if  one of the following two conditions holds

\noindent(i) there exists a sequence $\{(b_{n,1},b_{n,2}):b_{n,1},b_{n,2}\in[0,N_n-1]\}_{n=1}^\infty $ such that $0<\frac{b_{n,2}-b_{n,1}}{N_n}<l$ satisfying
	\[\sum_{n=1}^{\f}\frac{\# B_{n,2}^{l}}{\# B_n}<\f ,\]
where $B_{n,1}^{l}=B_n\cap [b_{n,1},b_{n,2}]$ and $B_{n,2}^{l}=B_n\backslash B_{n,1}^{l};$

\noindent(ii) there exists $c\in(0,1]$ and a sequence $\{(b_{n,1},b_{n,2}):b_{n,1},b_{n,2}\in[\frac{l}{2} N_n,(1-\frac{l}{2})N_n]\}_{n=1}^\infty $ such that
\[\frac{\# B_{n}^{l}}{\# B_n}\ge c,\]
for all $n\in \N$, where $B_{n}^{l}=B_n\cap [b_{n,1},b_{n,2}].$\\


Note that the sequence $\{( N_{n},B_{n})\}_{n=1}^\infty $ of general consecutive sets satisfying RBC always has a subsequence which satisfies  the partial concentration condition, see Example~\ref{ex_ncpt}, and the proof of Theorem~\ref{infinitely many admissible pairs}.

Finally, we  apply  the remainder bounded condition(RBC) and the partial concentration condition (PCC) together to obtain the spectrality of infinite convolution.
\begin{theorem}\label{gap-condition}
Let $\{( N_n,B_n)\}_{n=1}^\infty $ be a sequence of admissible pairs satisfying {\it RBC}.
If there exists a subsequence $\{( N_{n_k},B_{n_k})\}_{n=1}^\infty  $ satisfying  PCC, then the infinite convolution $\mu$  given by \eqref{infinite-convolution} exists and  is a spectral measure.
\end{theorem}

We say $\{( N_n,B_n)\}_{n=1}^\infty $ satisfies  \textit{digit bounded condition (DBC)}, if
\[
\sup_n\{N_n^{-1}|b|, \ b\in B_n\}<\f \quad \text{and} \quad \sup_n \# B_n<\f.
\]
This condition is frequently used in the spectral theory of infinite convolutions, see \cite{An-Fu-Lai-2019, Liu-Lu-Zhou}. One consequence of it is  that the  corresponding infinite convolution  exists. Moreover, if  $\sup_n\{N_n\}<\f$, that is, the sequence $\{( N_n,B_n)\}_{n=1}^\infty $ is chosen from finitely many admissible pairs, the infinite convolution exists and has compact support, see \cite{LMW-2023, Liu-Lu-Zhou}.
In the following theorem, we apply the digit bounded condition to part of the  admissible pairs.
\begin{theorem}\label{finitely many admissible pairs-1}
Given a sequence of admissible pairs $\{( N_n,B_n)\}_{n=1}^\infty $. If there exists a subsequence $\{(N_{n_k},B_{n_k})\}_{k=1}^\f$ satisfying PCC and {\it RBC},  and $\{( N_n,B_n)\}_{n\notin\{n_k\}} $ satisfies {\it DBC}, then the infinite convolution $\mu$  given by \eqref{infinite-convolution} exists and is a spectral measure.
\end{theorem}
If the subsequence in above theorem consists of  general consecutive sets, the partial concentration condition (PCC) may be omitted.
\begin{corollary}\label{finitely many admissible pairs}
Given a sequence of admissible pairs $\{( N_n,B_n)\}_{n=1}^\infty $. If there exists a subsequence $\{B_{n_k}\}_{k=1}^\f$ consisting of general consecutive sets satisfying RBC,  and $\{( N_n,B_n)\}_{n\notin\{n_k\}} $ satisfies  DBC, then the infinite convolution $\mu$  given by \eqref{infinite-convolution} exists and is a spectral measure.
\end{corollary}

 The rest of the paper is organized as follows. In Section~\ref{sec_exic}, we  provide the proof of Theorem \ref{existence}, and we cite some results on equi-positive families which are the key tools to prove the spectrality of measures. In Section~\ref{sec_im}, we give the proofs of Theorem~\ref{infinitely many admissible pairs} and Theorem \ref{gap-condition}. In Section \ref{sec_fm},  we prove Theorem \ref{finitely many admissible pairs-1} and Corollary \ref{finitely many admissible pairs}. Finally, we provide some examples in Section \ref{sec_ex}.

\section{Existence of infinite convolutions and equi-positive families}\label{sec_exic}
\subsection{Existence of infinite convolutions}
In this subsection, we prove Theorem \ref{existence} which provides some sufficient conditions for the convergence of infinite convolutions
 
 For a set $B\subset \Z$, we write $M_B(\xi)$ for the Fourier transform of the discrete measure $\delta_B$, that is,
\[M_B(\xi)=\frac{1}{\# B}\sum_{b\in B}e^{-2\pi i b \xi}. \]
Hence, the Fourier transform of the infinite convolution $\mu$ given by \eqref{infinite-convolution} is
\[\hat{\mu}(\xi)=\prod_{n=1}^{\infty}M_{B_n}\left(\frac{\xi}{N_1N_2\cdots N_n}\right).\]

We write $ \mathcal{P}(\R^d)$ for the set of all Borel probability measures on $ \R^d$.
For $\eta \in \mathcal{P}(\R^d)$, we define $$ E(\eta) = \int_{\R^d} x \D \eta(x), \qquad V(\eta) = \int_{\R^d} |x - c(\eta)|^2 \D \eta(x). $$
It is easy to check that
$$ V(\eta) = \int_{\R^d} |x|^2 \D \eta(x) - |E(\eta)|^2. $$
Given $r>0$, we define a new Borel probability measure $\eta_r$ by
\begin{equation}\label{eta_r}
	\eta_r(E) = \eta\big( E \cap D(r) \big) + \eta\big( \R^d \setminus D(r) \big)\delta_0(E),
\end{equation}
for every Borel subset $E \sse \R^d$,  where $D(r)$ is the closed ball with centre at $0$ and radius $r$.

To prove Theorem \ref{existence}, we need the following theorem which may be regarded as a distribution  version of Kolmogorov’s three series theorem, see \cite{CYS78,Jessen-Wintner-1935} for details.
\begin{theorem}\label{three-series-theorem}
	Let $\eta_1, \eta_2, \cdots$ be a sequence of Borel probability measures on $\R^d$.
	Fix a constant $r>0$, and let $\eta_{n,r}$ be defined by \eqref{eta_r} for each measure $\eta_n$.
	Then the infinite convolution $\eta_1 *\eta_2 *\cdots $ exists if and only if the following three series all converge:
	$$ (i) \sum_{n=1}^{\f} \eta_n\big( \R^d \setminus D(r)\big),\quad (ii) \sum_{n=1}^{\f} E(\eta_{n,r}), \quad (iii) \sum_{n=1}^{\f} V(\eta_{n,r}). $$
\end{theorem}

Now, we are ready to prove Theorem \ref{existence}, and the idea for the proof is from~\cite{Jessen-Wintner-1935}, where Jessen and Wintner studied the convergence of infinite convolutions in a general setting, and they gave sufficient and necessary conditions for the convergence of infinite convolutions.

\begin{proof}[Proof of Theorem \ref{existence}]
For simplicity, We write $\delta_n=\delta_{(N_1N_2\cdots N_n)^{-1}B_n}$ and $\delta_{n,r}(E)= \delta_n( E \cap D(r) ) + \delta_n( \R \setminus D(r) )\delta_0(E)$ where $D(r)$ is the closed ball with centre at $0$ and radius $r$.

To prove the existence of $\mu$,  by Theorem \ref{three-series-theorem}, it is sufficient to prove the following three series are convergent for $r=1$:
$$
(i) \sum_{n=1}^{\f} \delta_n\big( \R \setminus D(1)\big),\quad (ii) \sum_{n=1}^{\f} E(\delta_{n,1}), \quad (iii) \sum_{n=1}^{\f} V(\delta_{n,1}). 
$$

Since 
\begin{eqnarray*}
&&\sum_{n=1}^{\f} |E(\delta_{n,1})|\le\sum_{n=1}^{\f}\int_{\R} |x| \D \delta_{n,1}(x) <\f;\\
&&\sum_{n=1}^{\f} V(\delta_{n,1})\le\sum_{n=1}^{\f}\int_{\R} |x|^2 \D \delta_{n,1}(x)\le\sum_{n=1}^{\f}\int_{\R} |x| \D \delta_{n,1}(x)<\f,
\end{eqnarray*}
it is sufficient to prove the following two series are convergent for $r=1$:
$$
 (i) \sum_{n=1}^{\f} \delta_n\big( \R \setminus D(1)\big),\quad (ii) \sum_{n=1}^{\f}\int_{\R} |x| \D \delta_{n,1}(x). 
$$

If $\{( N_n,B_n)\}_{n=1}^\infty $ satisfies   RBC, i.e.$\sum_{n=1}^{\infty} \frac{\# B_{n,2}}{\# B_n} < \infty,$
where $B_{n,1}=B_n \cap \{0,1,\cdots,N_n-1\}$ and $B_{n,2}=B_n \backslash B_{n,1}$,  immediately  the following four series are convergent,
\begin{eqnarray*}
&&\sum_{n=1}^{\f} \delta_n\big( \R \setminus D(1)\big)\le\sum_{n=1}^{\infty} \frac{\# B_{n,2}}{\# B_n} < \infty; \\
&&\sum_{n=1}^{\f}\int_{\R} |x| \D \delta_{n,1}(x)\le\sum_{n=1}^{\f}\sum_{b\in B_{n},|b|<N_n}\bigg|\frac{(N_1N_2\cdots N_n)^{-1}b}{\# B_n}\bigg|\le1+\sum_{n=2}^{\f}\frac{1}{N_1N_2\cdots N_{n-1}} <\f.
\end{eqnarray*}
Therefore, the infinite convolution $\mu$ exists if $\{( N_n,B_n)\}_{n=1}^\infty $ satisfies  RBC.

If $\{( N_n,B_n)\}_{n=1}^\infty $ satisfies $\sum_{n=1}^{\f}\frac{\max\{|b|:b\in B_n\}}{N_1N_2\cdots N_n}<\f$,  there exists an integer $n_0>0$ such that  $\frac{\max\{|b|:b\in B_n\}}{N_1N_2\cdots N_n}<1$ for all $n>n_0$.  This implies that  $\delta_n( \R \setminus D(1))=0$ for all $n>n_0$. Hence we obtain that
\begin{eqnarray*}
&&\sum_{n=1}^{\f} \delta_n\big( \R \setminus D(1)\big)\le\sum_{n=1}^{n_0} \delta_n\big( \R \setminus D(1)\big)< \infty; \\
&&\sum_{n=1}^{\f}\int_{\R} |x| \D \delta_{n,1}(x)\le\sum_{n=1}^{\f}\sum_{b\in B_{n}}\bigg|\frac{(N_1N_2\cdots N_n)^{-1}b}{\# B_n}\bigg|\le \sum_{n=1}^{\f}\frac{\max\{|b|:b\in B_n\}}{N_1N_2\cdots N_n}<\f.
\end{eqnarray*}
Therefore, the infinite convolution $\mu$ exists if $\sum_{n=1}^{\f}\frac{\max\{b:b\in B_n\}}{N_1N_2\cdots N_n}<\f$.

Finally, suppose that there exists $r_0>0$ such that $\sum_{n=1}^{\infty} \frac{\# \{b:\ b\in B_n,\ |b|>N_1N_2\cdots N_nr_0\}}{\# B_n} = \infty$. Fix $r=r_0$, then we have
\begin{align*}
\sum_{n=1}^{\f} \delta_n\big( \R \setminus D(r_0)\big)=\sum_{n=1}^{\infty} \frac{\# \{b:\ b\in B_n,\ |b|>N_1N_2\cdots N_nr_0\}}{\# B_n} = \infty.
\end{align*}
Therefore, by Theorem \ref{three-series-theorem}, the infinite convolution $\mu$ does not exist.
\end{proof}

Recall that $\{( N_n,B_n)\}_{n=1}^\infty $ satisfies  digit bounded condition (DBC) if
\[
\sup_n\{N_n^{-1}|b|, \ b\in B_n\}<\f \quad \text{and} \quad \sup_n \# B_n<\f.
\]
In order to obtain the existence of infinite convolutions, we show that it is not necessary to require the whole   sequence satisfying RBC.
\begin{corollary}\label{existence_1}
		Given $\{( N_n,B_n)\}_{n=1}^\infty $ with a subsequence $\{(N_{n_k}B_{n_k})\}$ satisfying {\it RBC} and  $\{( N_n,B_n)\}_{n\notin\{n_k\}} $ satisfying {\it DBC}. Then the infinite convolution $\mu$ exists.
\end{corollary}

\begin{proof}[Proof of Corollary \ref{existence_1}]
	It follows from the same argument as  the proof of Theorem \ref{existence}.
\end{proof}

\subsection{Equi-positive families}
Equi-positive families are frequently used to study the spectrality of measures, and we recall some related results on equi-positive families, which are  important tools to study the spectrality of infinite convolutions in our work.

Let $\Phi$ be a set of probability measures supported on $\R$. We call $\Phi$ an equi-positive family if there exist $\epsilon_0>0$ and $\delta_0>0$ such that for all $x\in[0,1) $ and $\mu \in \Phi$, there exists an integer $k_{x,\mu }\in \Z$ such that
	\[| \hat{\mu}(x+y+k_{x,\mu })| \ge \epsilon_0,\]
	for all $|y|< \delta_0$, where $k_{x,\mu }=0$ for $x=0$.

The equi-positive family was used to study the spectrality of singular measures with compact support in \cite{An-Fu-Lai-2019,Dutkay-Haussermann-Lai-2019}, and it was then generalized in~\cite{LMW-2023} to the current version which is also applicable to infinite convolutions without compact support. The following theorem shows that equi-postive families imply spectrality which is a useful tool in studying the spectral theory of  these measures, see ~\cite{An-Fu-Lai-2019,Dutkay-Haussermann-Lai-2019,LMW-2023} for the proof.
\begin{theorem}\label{equi-positive}
	 Let $\{( N_n,B_n)\}_{n=1}^\infty $ be a sequence of admissible pairs. Suppose that the infinite convolution $\mu$ exists. If there exists a subsequence $\{\nu_{>n_k}\}_{k=1}^\infty$ which is an equi-positive family, then $\mu$ is a spectral measure.
\end{theorem}

In~\cite{Liu-Lu-Zhou}, Liu, Lu and Zhou gave an equivalent description for the equi-positive family  under   DBC,  and  we use the equivalence to prove the spectrality of infinite convolutions under   DBC. In fact, they require that $B_n\subseteq [0,+\f)$ for all $n\ge1$, however, when the DBC is satisfied, removing this requirement does not affect the conclusion.

\begin{theorem}\label{finitely many admissible pairs 1}
	 Given a sequence of admissible pairs $\{( N_n,B_n)\}_{n=1}^\infty $ satisfying {\it DBC}, let $\nu_{>n}$ be given by \eqref{def_nugn}. Then the following statements are equivalent:
	 \begin{enumerate}[(i)]
	 	\item   there exists an equi-positive sequence $\{(\nu_{>n_j})\}_{j=1}^\infty$;
	 	\item  at least one of the following two conditions is satisfied: $\# \{n:N_n>\# B_n\}=\f$ or $\gcd_{n>k}\{\gcd (B_n-B_n)\}=1$ for all $k\ge1$.
	 \end{enumerate}
\end{theorem}

\section{Infinite convolutions generated by infinitely many admissible pairs}\label{sec_im}

In this section, we study the spectrality of infinite convolutions generated by infinitely many admissible pairs. We first prove Theorem \ref{gap-condition} and then use  it to prove  Theorem~\ref{infinitely many admissible pairs}.

The following propositions are used in the proof of  Theorem \ref{gap-condition}.
\begin{proposition}\label{gap}
	Given $\theta\in(0,\pi)$. There exists a constant $r(\theta)\in (0,1)$ such that
	\[\Big|\frac{1}{m}\sum_{j=1}^{m}e^{-ix_j}\Big|\ge r(\theta),\]
for all $x_j\in[0,\theta]$, $j=1,\dots,m$, for all $m\in \N$.
\end{proposition}
\begin{proof}
Given $\theta\in(0,\pi)$, for all $x_j\in[0,\theta]$, $j=1,\dots,m$, it is straightforward that 
\begin{eqnarray*}
\Big|\frac{1}{m}\sum_{j=1}^{m}e^{-ix_j}\Big|  &=&  \Big|\frac{1}{m}\sum_{j=1}^{m}e^{-i(x_j-\frac{\theta}{2})}\Big|  \\
&\ge& \frac{1}{m}\sum_{j=1}^{m}\cos\bigg(x_j-\frac{\theta}{2}\bigg)  \\
&\ge&\cos\frac{\theta}{2}.
\end{eqnarray*}
Setting $r(\theta)=\cos\frac{\theta}{2}$, the concluions holds.
\end{proof}

In the rest of this paper, we always take
$$
r(\theta)=\cos\frac{\theta}{2}.
$$
 It is clear that  $0<r(\theta)< 1$ , and it is  continuous and  decreasing for  $\theta\in(0,\pi) $. 
\begin{proposition}\label{gap1}
	Given $\theta\in(0,\pi)$ and $c\in ( 0,1) $.  There exists a small $\Delta>0$ and  $\theta^{'}\in(\theta,\pi)$ such that for all $0<\delta <\Delta$, we have that 	
	$$(1+\delta)\theta \leq \theta^{'} \quad \textit{and } \quad  r(\theta^{'})\ge c\cdot r(\theta).
	$$
\end{proposition}
\begin{proof}
Since $r(\theta)$ is continuous and decreasing for  $\theta\in ( 0,\pi) $, the conclusion holds. 
\end{proof}

Next, we are ready to prove  Theorem~\ref{gap-condition}.
\begin{proof}[Proof of Theorem~\ref{gap-condition}]
Since $\{( N_n,B_n)\}_{n=1}^\infty $ satisfies remainder bounded condition(RBC),  by Theorem \ref{existence},  the infinite convolution exists. Note that $\{(N_{n_k},B_{n_k})\}_{k=1}^\infty  $ satisfies partial concentration condition (PCC),  by Theorem~\ref{equi-positive}, it is sufficient to show that $\{\nu_{>n_k-1}\}_{k=K}^\infty$ is an equi-positive family for some large integer $K$.

Recall that for each $k>1$,
	\[\nu_{>n_k-1}=\delta_{N_{n_k}^{-1}B_{n_k}}\ast\delta_{(N_{n_k}N_{n_k+1})^{-1}B_{n_k+1}}\ast\dots\ast\delta_{(N_{n_k}N_{n_k+1}\dots N_{n_k+j})^{-1}B_{n_k+j}}\ast\cdots,\]
and
\[
{\hat{\nu}}_{>n_k-1}(\xi)=\prod_{j=0}^{\infty}M_{B_{n_k+j}}\left(\frac{\xi}{N_{n_k}N_{n_k+1}\cdots N_{n_k+j}}\right)=M_{B_{n_k}}(\xi_0)M_{B_{n_{k+1}}}(\xi_1)\prod_{j=2}^{\infty}M_{B_{n_k+j}}(\xi_j),
\]	
where $\xi_j=\frac{\xi}{N_{n_k}N_{n_k+1}\dots N_{n_k+j}}$ for each $j\geq 0$.
The key is to give appropriate estimate for the lower bounds of $|M_{n_k}(\xi_0)|$, $|M_{n_k+1}(\xi_1)|$ and $\prod_{j=2}^{\infty}|M_{n_k+j}(\xi_j)|$, respectively.	

For all  $0<\delta<\frac{1}{4}$ and $\xi\in[-\frac{1}{2}-\delta,\frac{1}{2}+\delta]$,  it is obvious that
	\[|\xi_j|\le \frac{1}{N_{n_k+j}}\frac{| \xi| }{2^{j}}\le \frac{1}{N_{n_k+j}}\frac{3}{2^{j+2}}.\]

First, we estimate $\prod_{j=2}^{\infty}|M_{n_k+j}(\xi_j)|$.
		Let $B_n^{'}=B_n\pmod{N_n}$ for all $n\ge1$. It's clear that $B_n^{'}\subset \{0,1,\cdots,N_n-1\}.$	For all $b\in B_{n_k+j}^{'}$, we have that  $| 2\pi b \xi_j | \in [0,\frac{3\pi}{2^{j+1}}]$. This implies that  for $j\ge2$,
	\begin{align*}
		\left|M_{B_{n_k+j}}(\xi_j)\right|
		&=\Big|\frac{1}{\#B_{n_k+j}}\sum_{b\in B_{n_k+j}}e^{-2\pi i b\xi_j}\Big|\\
		&\ge\Big|\frac{1}{\#B_{n_k+j}}\sum_{b\in B_{n_k+j}^{'}}e^{-2\pi i b\xi_j}\Big|-\Big|\frac{1}{\#B_{n_k+j}}\sum_{b\in B_{n_k+j}^{'}}e^{-2\pi i b\xi_j}-\frac{1}{\#B_{n_k+j}}\sum_{b\in B_{n_k+j}}e^{-2\pi i b\xi_j}\Big|\\
		&\ge \Big|\frac{1}{\#B_{n_k+j}}\sum_{b\in B_{n_k+j}^{'}}\cos2\pi b \xi_j\Big|-\frac{2\#B_{n_k+j,2}}{\#B_{n_k+j}}\\
		&\ge \cos\frac{3\pi}{2^{j+1}}-\frac{2\#B_{n_k+j,2}}{\#B_{n_k+j}}.
	\end{align*}

	Since $\{(N_n, B_n)\}_{n=1}^\f$ satisfies RBC, i.e.  $\sum_{n=1}^{\infty} \frac{\# B_{n,2}}{\# B_n} < \infty$, combining it with $\sum_{j=2}^{\f}( 1-\cos\frac{3\pi}{2^{j+1}})<\infty$,	we have
	\[\lim_{j\to \f}\cos\frac{3\pi}{2^{j+1}}-\frac{2\#B_{n_k+j,2}}{\#B_{n_k+j}}-1=0,\]
	for all $k\in \N$, and it implies that
	\begin{eqnarray}\label{7}
		\lim_{j\to \f}\frac{\ln ( \cos\frac{3\pi}{2^{j+1}}-\frac{2\#B_{n_k+j,2}}{\#B_{n_k+j}}) }{\cos\frac{3\pi}{2^{j+1}}-\frac{2\#B_{n_k+j,2}}{\#B_{n_k+j}}-1}=1.
	\end{eqnarray}
Hence there exist constants $K_1>0$ and $C>0$ such that for all $n>K_1$ and $k>K_1$,
	\[\frac{2\#B_{n,2}}{\#B_n}<\cos\frac{3\pi}{8},\]
		\[\sum_{j=2}^{\f}\ln \left( \cos\frac{3\pi}{2^{j+1}}-\frac{2\#B_{n_k+j,2}}{\#B_{n_k+j}}\right)>-C.\]
Therefore,  for $j\geq 2$,  we obtain the lower bound
	\begin{equation}\label{MB>1}
		\prod_{j=2}^{\infty}\left|M_{n_k+j}(\xi_j)\right|\ge\prod_{j=2}^{\f}\left( \cos\frac{3\pi}{2^{j+1}}-\frac{2\#B_{n_k+j,2}}{\#B_{n_k+j}}\right)>e^{-C},	
	\end{equation}
	for all $k>K_1.$
	
Next, we estimate $|M_{n_k+1}(\xi_1)|$. Since $\sum_{n=1}^{\infty} \frac{\# B_{n,2}}{\# B_n} < \infty,$
	there exists $K_2\in \N$ such that for all $k>K_2$,
	\[\frac{\#B_{n_k+1,2}}{\#B_{n_k+1}}<\frac{1}{4}\cdot r\Big(\frac{3\pi}{4}\Big).\]
	Thus, by Proposition \ref{gap},
	\begin{equation}\label{MB=1}
		\begin{split}
			&\left|M_{B_{n_k+1}}(\xi_1)\right|
			=\Big|\frac{1}{\#B_{n_k+1}}\sum_{b\in B_{n_k+1}}e^{-2\pi i b\xi_1}\Big|\\
			&\ge\Big|\frac{1}{\#B_{n_k+1}}\sum_{b\in B_{n_k+1}^{'}}e^{-2\pi i b\xi_1}\Big|-\Big|\frac{1}{\#B_{n_k+1}}\sum_{b\in B_{n_k+1}^{'}}e^{-2\pi i b\xi_1}-\frac{1}{\#B_{n_k+1}}\sum_{b\in B_{n_k+1}}e^{-2\pi i b\xi_1}\Big|\\
			&\ge r\Big(\frac{3\pi}{4}\Big)-\frac{2\#B_{n_k+1,2}}{\#B_{n_k+1}}\\
			&=\frac{1}{2}\cdot r\Big(\frac{3\pi}{4}\Big),
		\end{split}
	\end{equation}
	for all $k>K_2$.
	
Finally, we need accurate estimate for $M_{B_{n_k}}(\xi_0)$. Since $\{(N_{n_k},B_{n_k})\}_{k=1}^\infty  $ satisfies  PCC, which contains two different cases, we need to consider both cases respectively.

Case(i),  there exists $\{(b_{n_k,1},b_{n_k,2}):b_{n_k,1},b_{n_k,2}\in[0,N_{n_k}-1]\}_{k=1}^\infty$ such that $0<\frac{b_{n_k,2}-b_{n_k,1}}{N_{n_k}}<l$ and $\sum_{k=1}^{\f}\frac{\# B_{n_k,2}^{l}}{\# B_{n_k}}<\f ,$	where
$B_{n_k,1}^{l}=B_{n_k}\cap [b_{n_k,1},b_{n_k,2}]$, $B_{n_k,2}^{l}=B_{n_k}\backslash B_{n_k,1}^{l}$.

Since $0<\pi l<\pi$ and $r(\theta)<1$,  we have that  $r(\pi l)<1$ and  $\frac{2}{4-r(\pi l)} <1$. by Proposition \ref{gap1}, there exists $\delta'_1>0$ and  $l'\in(l,1)$ such that  $0<\frac{b_{n_k,2}-b_{n_k,1}}{N_{n_k}}(1+\delta) <l'$ and $r(\pi l')\ge \frac{2}{4-r(\pi l)}\cdot r(\pi l),$ for all $0<\delta<\delta'_1$.
For all $k\in\N$, by  Proposition \ref{gap}, we have
\[\Big|\frac{1}{\# B_{n_k,1}^{l}}\sum_{b\in B_{n_k,1}^{l}}e^{-2\pi i b\xi_0}\Big|=
	\Big|\frac{1}{\# B_{n_k,1}^{l}}\sum_{b\in B_{n_k,1}^{l}}e^{2\pi (b-b_{n_k,1})\xi_0i}\Big|\ge r(\pi l'),\]
for all $ \xi\in[-\frac{1}{2}-\delta,\frac{1}{2}+\delta]$, where $0<\delta<\delta_1'$ and $\xi_0=\frac{\xi}{N_{n_k}}$.

   Since $\sum_{k=1}^{\f}\frac{\# B_{n_k,2}^{l}}{\# B_{n_k}}$ converges, there exists $K'_3>0$ such that for all $k>K'_3$,
\[
\frac{\# B_{n_k,2}^{l}}{\# B_{n_k}}<\frac{1}{4}\cdot r(\pi l) \quad \text{and} \quad \frac{\# B_{n_k,1}^{l}}{\# B_{n_k}}>1-\frac{1}{4}\cdot r(\pi l) .
\]
Thus,  for all $k>K'_3$, by Proposition \ref{gap}, we obtain that
\begin{equation}\label{=0}
   	\begin{split}
   	\left|M_{B_{n_k}}(\xi_0)\right|
   	&\ge\Big|\frac{1}{\#B_{n_k}}\sum_{b\in B_{n_k,1}^{l}}e^{-2\pi i b\xi_0}\Big|-\Big|\frac{1}{\#B_{n_k}}\sum_{b\in B_{n_k,1}^{l}}e^{-2\pi i b\xi_0}-\frac{1}{\#B_{n_k}}\sum_{b\in B_{n_k}}e^{-2\pi i b\xi_0}\Big|\\
   	&\ge\Big(\frac{\# B_{n_k,1}^{l}}{\# B_{n_k}}\Big)\cdot r(\pi l')-\frac{\# B_{n_k,2}^{l}}{\# B_{n_k}}\\
   	&\ge \Big(1-\frac{1}{4}\cdot r(\pi l)\Big)\cdot\frac{2}{4-r(\pi l)}\cdot r(\pi l)-\frac{1}{4}\cdot r(\pi l)\\
   	&=\frac{1}{4}\cdot r(\pi l),
   	\end{split}
   \end{equation}
for all $ \xi\in[-\frac{1}{2}-\delta,\frac{1}{2}+\delta]$, where $0<\delta<\delta_1'$, and $\xi_0=\frac{\xi}{N_{n_k}}$.


Case(ii), for $\{(N_{n_k},B_{n_k})\}_{k=1}^\infty$, there exists $c\in(0,1]$ and a sequence $\{(b_{n_k,1},b_{n_k,2}):b_{n_k,1},b_{n_k,2}\in   [\frac{l}{2}N_{n_k},(1-\frac{l}{2})N_{n_k}]    \}_{n=1}^\infty $ such that $\frac{\# B_{n_k}^{l}}{\# B_{n_k}}\ge c,$ for all $k\in \N$, where $B_{n_k}^{l}=B_{n_k}\cap [b_{n_k,1},b_{n_k,2}].$

It is clear that there exists $\delta''_1>0$ and $0<l''<1$  such that for all  $\delta<\delta''_1$, for all $|\xi|\in[\frac{1}{4},\frac{1}{2}+\delta]$ and all $b\in B_{n_k}^{l}$,  we have that  $ 2\pi b \frac{\xi}{N_{n_k}}\in(\pi \frac{l''}{2},\pi-\frac{l''}{2},)$.   Since $\frac{\# B_{n_k}^{l}}{\# B_{n_k}}\ge c$  and $B'_{n_k}= B_{n_k}\pmod{N_{n_k}}$,
$$
\# B'_{n_k}\setminus \{b\in B'_{n_k}:2\pi b \frac{\xi}{N_{n_k}}\in[0,\pi]\}\le (1-c) \cdot\# B_{n_k},
$$
for all  $\delta<\delta''_1$ and all $|\xi|\in[\frac{1}{4},\frac{1}{2}+\delta]$. It is also true that  $2\pi b \frac{\xi}{N_{n_k}}\in(0,\pi+2\pi\delta)$ for all $|\xi|\in[\frac{1}{4},\frac{1}{2}+\delta]$ and all $b\in B'_{n_k}$. Thus there exists $\delta'''_1\in(0,\delta''_1)$ such that $c_\delta=c\sin\frac{\pi l''}{2}+(1-c)\sin(\pi+2\pi\delta)>0$ for all $0<\delta<\delta'''_1$.

Therefore the elements of $B'_{n_k}$ is classified into three groups according to the location of $ 2\pi b \frac{\xi}{N_{n_k}}$, for all $|\xi|\in[\frac{1}{4},\frac{1}{2}+\delta]$ and all $0<\delta<\delta'''_1$:
\begin{eqnarray*}
A_1&=&\{b\in B'_{n_k}: 2\pi b \frac{\xi}{N_{n_k}}\in[0,\frac{\pi l''}{2}]\cup[\pi-\frac{\pi l''}{2},\pi]\} ,  \\
A_2&=&\{b\in B'_{n_k}: 2\pi b \frac{\xi}{N_{n_k}}\in(\frac{\pi l''}{2},\pi-\frac{\pi l''}{2})\},  \\
A_3&=&\{b\in B'_{n_k}: 2\pi b \frac{\xi}{N_{n_k}}\in(\pi, \pi+2\pi\delta]\}.
\end{eqnarray*}
Note that $\# A_2\geq \# B_{n_k}^{l}\geq c\# B_{n_k}$, $\# A_3\leq (1-c) \cdot\# B_{n_k}$. Immediately, we have the following estimate
\begin{eqnarray*}
\Big|\frac{1}{\#B_{n_k}}\sum_{b\in B'_{n_k}}\sin2\pi  b\frac{\xi}{N_{n_k}}\Big| &=&\frac{1}{\#B_{n_k}}\Big|\sum_{b\in A_1}\sin2\pi  b\frac{\xi}{N_{n_k}}+\sum_{b\in A_2}\sin2\pi  b\frac{\xi}{N_{n_k}}+\sum_{b\in A_3}\sin2\pi  b\frac{\xi}{N_{n_k}}\Big| \\
&\geq &\frac{1}{\#B_{n_k}}\Big|\sum_{b\in A_2}\sin2\pi  b\frac{\xi}{N_{n_k}}+\sum_{b\in A_3}\sin2\pi  b\frac{\xi}{N_{n_k}}\Big| \\
&\geq &     \frac{1}{\#B_{n_k}}\Big|c\# B_{n_k} \sin\frac{\pi l''}{2}+(1-c) \cdot\# B_{n_k}\sin(\pi+2\pi\delta)\Big|           \\
&=&c_\delta,
\end{eqnarray*}
for all $|\xi|\in[\frac{1}{4},\frac{1}{2}+\delta]$ and all $0<\delta<\delta'''_1$.

Recall that $\xi_0=\frac{\xi}{N_{n_k}}$, we have
 \begin{align*}
 \big|M_{B_{n_k}}\left(\xi_0\right)\big|
 &\ge\Big|\frac{1}{\#B_{n_k}}\sum_{b\in B'_{n_k}}e^{-2\pi i b\frac{\xi}{N_{n_k}}}\Big|-\Big|\frac{1}{\#B_{n_k}}\sum_{b\in B'_{n_k}}e^{-2\pi i b\frac{\xi}{N_{n_k}}}-\frac{1}{\#B_{n_k}}\sum_{b\in B_{n_k}}e^{-2\pi i b\frac{\xi}{N_{n_k}}}\Big|\\
 &\ge \Big|\frac{1}{\#B_{n_k}}\sum_{b\in B'_{n_k}}\sin2\pi  b\frac{\xi}{N_{n_k}}\Big|-\frac{2\#B_{n_k,2}}{\#B_{n_k}} \\
 &\ge c_\delta -\frac{2\#B_{n_k,2}}{\#B_{n_k}},\\	
 \end{align*}
for all $|\xi|\in[\frac{1}{4},\frac{1}{2}+\delta]$, where $\delta<\delta'''_1$. It is clear that $2\pi i b\frac{\xi}{N_{n_k}}\in[0,\frac{\pi}{2}]$ for all $|\xi|\in[0,\frac{1}{4}]$ and all $b\in B'_{n_k}$, and by Proposition \ref{gap}, it follows that
     \begin{align*}
  	\big|M_{B_{n_k}}\left(\xi_0\right)\big|
  	&\ge\Big|\frac{1}{\#B_{n_k}}\sum_{b\in B'_{n_k}}e^{-2\pi i b\frac{\xi}{N_{n_k}}}\Big|-\Big|\frac{1}{\#B_{n_k}}\sum_{b\in B'_{n_k}}e^{-2\pi i b\frac{\xi}{N_{n_k}}}-\frac{1}{\#B_{n_k}}\sum_{b\in B_{n_k}}e^{-2\pi i b\frac{\xi}{N_{n_k}}}\Big|\\
  	&\ge r(\frac{\pi}{2})-\frac{2\#B_{n_k,2}}{\#B_{n_k}},
  \end{align*}
for all $|\xi|\in[0,\frac{1}{4}]$.

Let $\epsilon'_\delta=\frac{1}{3}\min\{c_\delta,r(\frac{\pi}{2})\}>0$. Since $\sum_{n=1}^{\infty} \frac{\# B_{n,2}}{\# B_n} < \infty$, there exists $K''_3>0$ such that for all $k>K''_3$, $\frac{\#B_{n_k,2}}{\#B_{n_k}}<\epsilon'_\delta.$ It follows that
    \begin{equation}\label{=0_1}
	    \big|M_{B_{n_k}}(\xi_0)\big|\ge \epsilon'_\delta,
    \end{equation}
for all $k \ge K''_3$ and all $\xi \in [-\frac{1}{2}-\delta,\frac{1}{2}+\delta]$, where $\delta<\delta'''_1$.

Therefore, by choosing  $K_3=\max\{K_3',K_3'' \}$, $\epsilon_\delta=\min\{\epsilon_\delta', \frac{1}{4} r(\theta)\}>0$ and $\delta_1=\min\{\delta'_1,\delta'''_1\}$,  we  have that
\begin{equation}\label{MB0}
    	\big|M_{B_{n_k}}(\xi_0)\big|\ge \epsilon_\delta,
\end{equation}
for all $k \ge K_3$ and all $\xi \in [-\frac{1}{2}-\delta,\frac{1}{2}+\delta]$, where $\delta<\delta_1$.

Finally, by combining \eqref{MB>1}, \eqref{MB=1} and \eqref{MB0} together,  let
\[
\delta_0 =\min\set{\frac{1}{4},\delta_1}, \quad \epsilon_0=\epsilon_{\delta_0}\cdot\frac{1}{2}\cdot r\left(\frac{3\pi}{4}\right)\cdot e^{-C'},\  K=\max\set{K_1,K_2,K_3},
\]
and  for all $\xi\in[-\frac{1}{2}-\delta_0,\frac{1}{2}+\delta_0]$ and all $k>K$, we obtain
\[
|\hat{\nu}_{>n_k-1}(\xi)|=|M_{n_k}(\xi_0)||M_{n_k+1}(\xi_1)|\prod_{j=2}^{\infty}|M_{n_k+j}(\xi_j)|\ge\epsilon_0.
\]
Hence, for all $k>K$ and all $x\in[0,1)$, setting $k_x=0$ for $x\in[0,\frac{1}{2})$ and $k_x=-1$ for $x\in[\frac{1}{2},1)$, it follows that
    \[
    |\hat{\nu}_{>n_k-1}(x+y+k_x)|\ge\epsilon_0,
    \]
all $|y|<\delta_0$.  This implies that  $\{\nu_{>n_k-1}:k>K\}$ is an equi-positive family. By Theorem \ref{equi-positive}, the infinite convolution $\mu$ is a spectral measure.
\end{proof}

Finally, we give the proof of Theorem \ref{infinitely many admissible pairs}.
\begin{proof}[Proof of Theorem \ref{infinitely many admissible pairs}]
	 Assume the subsequence $\{(N_{n_k},B_{n_k})\}_{k=1}^\infty $ consists of general consecutive sets. For all $k\ge1$, let $B'_{n_k}=B_{n_k}\pmod{N_{n_k}}$. Then $B'_{n_k}=\{0,1,\dots,\# B_{n_k}-1\}$. Note that $\# B_{n_k}|N_{n_k}$ since $(N_{n_k},B_{n_k}) $ is  an admissible pair.
	
	 If $\#\{k\in\N:\# B_{n_k}=N_{n_k}\}=\f$, we may assume the subsequence $\{(N_{n_k},B_{n_k})\}_{k=1}^\infty $ consists of general consecutive sets satisfying $\# B_{n_k}=N_{n_k}$ for all $k\in\N$. Choosing $l=\frac{1}{2}$, we have
	 \[\frac{\#\{b\in B'_{n_k}:b\in[\frac{N_{n_k}}{4},\frac{3N_{n_k}}{4}]\}}{N_{n_k}}>\frac{1}{8},\]
	 for all $k\geq 1$. 	 Since $\{(N_n,B_n)\}_{n=1}^\infty $ satisfies  RBC, we have that $\sum_{n=1}^{\infty} \frac{\# B_{n,2}}{\# B_n} < \infty$. For  sufficiently large $k$, we always have
	 \[
        \frac{\#B_{n}^{l}}{\#B_{n_k}}\ge\frac{\#\{b\in B'_{n_k}:b\in[\frac{N_{n_k}}{4},\frac{3N_{n_k}}{4}]\}}{N_{n_k}}-\frac{\# B_{n_k,2}}{\# B_{n_k}}>\frac{1}{16}.
      \]
Therefore $\{(N_{n_k},B_{n_k})\}_{k=1}^\infty $ has a subsequence satisfying  PCC.
	
	 If $\#\{k\in\N:\# B_{n_k}<N_{n_k}\}=\f$, we may assume the subsequence $\{(N_{n_k},B_{n_k})\}_{k=1}^\infty $ consists of general consecutive sets satisfying $\# B_{n_k}<N_{n_k}$, for all $k\geq 1$. Obviously $\#B_{n_k}\le\frac{N_{n_k}}{2}$. 	
Choosing $l=\frac{1}{2}$ and $(b_{n_k,1},b_{n_k,2})=(0,\frac{N_{n_k}}{2})$, for all $k\geq 1$, we have
	 $B_{n_k,1}^{l}=B_{n_k}\cap [0,\frac{N_{n_k}}{2}]$ and  $B_{n_k,2}^{l}=B_{n_k}\backslash B_{n_k,1}^{l}=B_{n_k,2}$.
Since $\{(N_n,B_n)\}_{n=1}^\infty $ satisfies  RBC, we have
\[
\sum_{k=1}^{\f}\frac{\# B_{n_k,2}^{l}}{\# B_{n_k}}=\sum_{k=1}^{\infty} \frac{\# B_{n_k,2}}{\# B_{n_k}}<\f,
\]
which implies that $\{(N_{n_k},B_{n_k})\}_{k=1}^\infty $ satisfies  PCC.
	
Hence $\{(N_n,B_n)\}_{n=1}^\infty $ always has a subsequence satisfying   PCC, by Theorem \ref{gap-condition}, the infinite convolution $\mu$ is a spectral measure.
\end{proof}

\section{spectral measures generated by special subsequences}\label{sec_fm}

In this section, we prove Theorem \ref{finitely many admissible pairs-1} and Corollary \ref{finitely many admissible pairs}.
\begin{proof}[Proof of Theorem \ref{finitely many admissible pairs-1}]
Since the subsequence $\{(N_{n_k},B_{n_k})\}_{k=1}^\f$ of $\{( N_n,B_n)\}_{n=1}^\infty $  satisfies   RBC, and the rest $\{( N_n,B_n)\}_{n\notin\{n_k\}} $ satisfies   DBC, by Corollary \ref{existence_1}, the infinite convolution exists. It remains to show that $\mu$ is a spectral measure.

Since $\{( N_n,B_n)\}_{n\notin\{n_k\}} $ satisfies  DBC, i.e.
\[
\sup_{n\notin{n_k}}\{N_n^{-1}|b|, \ b\in B_n\}<\f \quad \text{and} \quad \sup_{n\notin{n_k}} \# B_n<\f,
\]
there exists a real $N>0$ such that $\sup\{N_n^{-1}b :  \ b\in B_n, n\notin\{n_k\}_{k=1}^\f\}\leq N$. If $N<1$,  we have that $B_{n,2}=\emptyset$, and it implies that $\{( N_n,B_n)\}_{n=1}^\infty $ satisfies   RBC since $\{(N_{n_k},B_{n_k})\}_{k=1}^\f$ satisfies  RBC and   PCC. By Theorem \ref{gap-condition},  $\mu$ is a spectral measure, and the conclusion holds.

For  $N\ge 1$, we need to consider the number of elements  such that $N_{n_k}\geq 2N$ in the subsequence $\{( N_{n_k},B_{n_k})\}_{k=1}^\infty $,  and we divide it into two cases.

Case (i): there are only finitely many elements such that $N_{n_k}\ge 2N$. Hence there exists $M>0$ such that $N_{n_k}<2N$ for all $k>M$. Since the sequence $\{( N_n,B_n)\}_{n\notin\{n_k\}} $ satisfying RBC implies that $\sum_{k=1}^{\infty} \frac{\# B_{n_k,2}}{\# B_{n_k}} < \infty,$ there exists $M'>M$ such that $B_{n_k,1}=B_{n_k}$, for all $k>M'$. Otherwise
	\[\sum_{k=1}^{\infty} \frac{\# B_{n_k,2}}{\# B_{n_k}}\ge\sum_{k>M, B_{n_k,1}\ne B_{n_k}}   \frac{1}{2N} = \infty,\]
	this contradicts the fact that $\{(N_{n_k}B_{n_k})\}$  satisfies   RBC. Therefore, $\{(N_{n},B_{n})\}_{n>n_{M'}}$ satisfies  DBC.

For all $k>M'$, we have either  $N_{n_k}>\#B_{n_k}$ or $B_{n_k}=\{0,1,\dots,N_{n_k}-1\}$ , it follows that $\# \{n>n_{M'}:N_n>\# B_n\}=\f$ or $\gcd_{n>n_k}\{\gcd (B_n-B_n)\}=1$.
	By Theorem \ref{finitely many admissible pairs 1}, there exists a subsequence $\{(\nu_{>n_j})\}_{j=1}^\infty$ which is an equi-positive family, and by Theorem \ref{equi-positive}, $\mu$ is a spectral measure.

Case (ii):  there are infinitely many elements such that $N_{n_k}\ge 2N$  in the subsequence $\{( N_{n_k},B_{n_k})\}_{k=1}^\infty $. There exists $M>0$ such that for all $k>M$, we have $B_{n_k,1}=B_{n_k}$ if $N_{n_k}<2N$. Otherwise for all $l\in\N$, there exists an integer $k_l> l$ such that $N_{n_{k_l}}<2N$ and $B_{n_{k_l},2}\ne \emptyset$, This implies that
\[\sum_{k=1}^{\infty} \frac{\# B_{n_k,2}}{\# B_{n_k}}\ge\sum_{l=1}^\f \frac{\# B_{n_{k_l},2}}{\# B_{n_{k_l}}} \ge \sum_{l=1}^\f\frac{1}{2N} = \infty,\]
which contradicts the fact that $\{(N_{n_k}B_{n_k})\}$  satisfies   RBC.

Therefore, there exists a subsequence $\{(N_{m_k},B_{m_k})\}_{k=1}^\f $ of $\{(N_{n_k},B_{n_k})\}_{k=M}^\f $  satisfies   RBC and  PCC such that  $N_{m_k}\ge 2N$ for all $k>1$, and $N_n<2N$ for all $n\in \{n_k\}_{k=M+1}^\f\backslash \{m_k\}_{k=1}^\f$.

	Similar to the proof of Theorem \ref{gap-condition}, we estimate the Fourier transform of $\nu_{>m_k-1}$, that is,
\[
|\hat{\nu}_{>m_k-1}(\xi)|=|M_{m_k}(\xi_0)|\prod_{j=1}^{\infty}|M_{m_k+j}(\xi_j)|.
\]
By the same argument as  \eqref{=0} and \eqref{=0_1}, there exist  $K'>0$, $\epsilon_1>0$ and $\delta_1>0$, such that
$$
\big|M_{B_{m_k}}(\xi_0)\big|\ge \epsilon_1,
$$
for all $\xi\in[-\frac{1}{2}-\delta_1,\frac{1}{2}+\delta_1]$ and $k>K'$.

Next, we estimate $\prod_{j=1}^{\infty}|M_{m_k+j}(\xi_j)|$.
Since $N_{n_k}<2N$ for  all  $n_k \in \{n_k\}_{k=M+1}^\f\backslash \{m_k\}_{k=1}^\f$, we have $B_{n_k,1}=B_{n_k}$ which implies that $N_{n_k}^{-1}b<1$ for all $b\in B_{n_k}$. Combining it with $\sup_{n\notin\{n_k\}_{k=1}^\f}\{N_n^{-1}b, \ b\in B_n\}\le N$, it follows that

$$
\sup_{n\notin\{m_k\}_{k=1}^\f \ , \ n>n_M}\{N_n^{-1}b, \ b\in B_n\}\le N.
$$
We write that
$$
	B''_n=\left\{
	\begin{array}{ll}
		\ B_n\pmod{N_n} & \text{if} \ n\in\{m_k\}_{k=1}^\infty ;\\
		\ B_n & \text{if} \ n\notin\{m_k\}_{k=1}^\infty,
	\end{array}  \right.
$$
	and it is clear that $N\ge \max\{\frac{b}{N_n}:b\in B''_n\}$ for all $n>n_M$. Hence
	\[\sum_{n>n_M}\frac{\# B_n\backslash B''_n}{\#B_n}=\sum_{k=1}^{\f}\frac{\# B_{{m_k},2}}{\#B_{m_k}}<\f.\]	
For all $ \xi\in[-\frac{3}{4},\frac{3}{4}]$ and $m_k>n_M$, setting $\xi_j=\frac{\xi}{N_{m_k}N_{m_{k+1}}\cdots N_{m_{k+j}}}$, $j=1,2,\ldots$, we have
	\begin{align*}
		|2\pi b \xi_j|
		&\le\Big| \frac{2\pi N\xi}{2NN_{m_{k+1}}\cdots N_{m_{k+j-1}}}\Big|\\
		&\le \frac{3\pi}{2^{j+1}},
	\end{align*}
	for all $b\in B''_{m_{k+j}}$.
	By the same argument as  \eqref{MB>1}, there exist $K''>0$ and $\epsilon_2>0$ such that
	\[\prod_{j=1}^{\infty}|M_{m_k+j}(\xi_j)| >\epsilon_2,\]
	for all $\xi\in[-\frac{3}{4},\frac{3}{4}]$ and $k>K''$.
	
Let
	\[\epsilon_0=\epsilon_1\cdot\epsilon_2,\quad K=\max\set{K',K''},\quad 0<\delta_0< \min\set{\delta_1,\frac{1}{4}} .\]
We obtain that   	\[|\hat{\nu}_{>n_k-1}(\xi)|=|M_{m_k}(\xi_0)|\prod_{j=1}^{\infty}|M_{m_k+j}(\xi_j)|\ge\epsilon_0,\]
for all $\xi\in[-\frac{1}{2}-\delta_0,\frac{1}{2}+\delta_0]$ and $k>K$.

Setting $k_x=0$ for $x\in[0,\frac{1}{2})$ and $k_x=-1$ for $x\in[\frac{1}{2},1)$,  we have
	\[|\hat{\nu}_{>m_k-1}(x+y+k_x)|\ge\epsilon_0,\]
for all $k>K$, $x\in[0,1)$ and $|y|<\delta_0$.	
Hence $\{\nu_{>m_k-1}:k>K\}$ is an equi-positive family, and by Theorem \ref{equi-positive}, the infinite convolution $\mu$ is a spectral measure.
\end{proof}

\begin{proof}[Proof of Corollary \ref{finitely many admissible pairs}]
By the proof of Theorem \ref{infinitely many admissible pairs}, the subsequence $\{( N_{n_k},B_{n_k})\}_{k=1}^\infty $ of general consecutive sets satisfying   RBC implies that $\{( N_{n_k},B_{n_k})\}_{k=1}^\infty$ has a subsequence satisfying   PCC. By Theorem~\ref{finitely many admissible pairs-1}, the infinite convolution $\mu$ is a spectral measure.
\end{proof}

\section{ Examples}\label{sec_ex}

In this section, we give some examples to illustrate our main conclusions. The first example shows that the infinite convolution of admissible pairs may not exist.
\begin{example}\label{not converge}
For each integer $n\geq 1$, let $N_n =2$ and $B_n=\set{ 0,2^n+1 }.$ Then $B_n$ is a general consecutive set. By choosing $L_n=\set{ 0,1 }$, it is clear that $\{( 2,B_n)\}_{n=1}^\infty $ is a sequence of admissible pairs. Let $\mu_n$ be given by \eqref{def_mun}.

By Theorem \ref{existence}, since $B_n\setminus[-2^{n},2^{n}]=\{2^n+1\}$ and
    \begin{equation*}
    	\sum_{n=1}^{\infty}\frac{\# B_n\setminus[0,2^{n}]}{\# B_n}=\sum_{n=1}^{\infty}\frac{1}{2}=\f,
    \end{equation*}
the sequence $\{\mu_n\}$ does not converge weakly, that is,  the infinite convolution $\mu$ does not exist.
\end{example}

It happens that infinite convolutions are  not spectral measures. In the next example, we show that the infinite convolutions are not spectral measures even if the admissible pairs are general consecutive sets.
\begin{example}\label{ex_NE}
For each integer $n\geq 1$, let $N_n =2$ and $B_n=\set{ 0,1 } \text{or} \set{0,3}
.$ It is clear that  $B_n$ is a general consecutive set, and $\{( 2,B_n)\}_{n=1}^\infty $ is a sequence of admissible pairs.  Then the infinite convolution $\mu$ is not a spectral measure if and only if $$0<\#\{n\in\N:B_n=\set{ 0,1 }\}<\f.$$

If $\#\{n\in\N:B_n=\set{ 0,1 }\}=\f$, it contains a subsequence satisfying   RBC, and the rest satisfies  DBC. By Corollary \ref{finitely many admissible pairs}, $\mu$ is a spectral measure.

If  $\#\{n\in\N:B_n=\set{ 0,1 }\}=0$, the infinite convolution $\mu=\frac{1}{3}\mbb{L}|_{[0,3]}$, where $\mbb{L}|_{[0,3]}$ denotes the Lebesgue measure restricted on the interval $[0,3]$. Thus $\mu$ is a spectral measure.

If $0<\#\{n\in\N:B_n=\set{ 0,1 }\}<\f$, there exists $n_0\in\N$ such that $B_{n_0}=\set{ 0,1 }$ and $B_n=\set{ 0,3}$ for all $n>n_0$. Then $\nu_{>n_0-1}=\delta_{\frac{1}{2}\set{ 0,1 }}\ast\delta_{\frac{1}{4}\set{ 0,3 }}\ast\delta_{\frac{1}{8}\set{ 0,3 }}\ast\cdots=\frac{1}{3}\mbb{L}|_{[0,2]}+\frac{1}{3}\mbb{L}|_{[\frac{1}{2},\frac{3}{2}]}$ is absolutely continuous but not uniform on its support. For admissible pair $(N,B)$, we write $\rho(\cdot)=\nu_{>n_0-1}(N\cdot)$. Obviously, $\delta_{N^{-1}B}\ast\rho$ is still absolutely continuous but not uniform on its support. Thus $\mu$ is absolutely continuous but not uniform on its support. However, absolutely continuous spectral measures are  uniform distributed on their support, see \cite{Dutkay-Lai-2014} for details. Therefore, $\mu$ is not a spectral measure.

Note that $\{( 2,B_n)\}_{n=1}^\infty $ does not satisfy  PCC in the latter case since for all $n>n_0$ and $l\in (0,\frac{1}{2})$, $B^{l}_{n,2}=\{0,3\} \ \text{or} \ \{3\}$ and $B_n^l=\varnothing $.
\end{example}

Finally, we provide an example that the infinite convolution is a spectral measure, but it is not compactly supported.
\begin{example}\label{ex_ncpt}
Let $\{n_k\}_{k=1}^\infty $ be a strictly increasing sequence of positive integers. For each $n\geq 1$,  we write
$$
N_n=
    \begin{cases}
    	2(n+1)^2 & \text{if} \ n\in \{n_k\}_{k=1}^\infty , \\
    	9 & \text{if}  \ n\notin \{n_k\}_{k=1}^\infty ,
    \end{cases}
$$
and
$$
B_n=
    \begin{cases}
    	\set{ 0,1,\cdots, (n+1)^2-2, (n+1)^2 -1 + N_1 N_2 \cdots N_n } & \text{if} \ n\in \{n_k\}_{k=1}^\infty , \\
    	\set{0,1,5} & \text{if}  \ n\notin \{n_k\}_{k=1}^\infty .
    \end{cases}
$$
It is clear that $\{( N_n,B_n)\}_{n=1}^\infty $ is a sequence of admissible pairs by choosing
$$
L_n=
   \begin{cases}
	\set{ 0,2,\cdots, 2(n+1)^2-4, 2(n+1)^2 -2 } & \text{if} \ n\in\{n_k\}_{k=1}^\infty ,\\
	\set{0,3,6} & \text{if}  \ n\notin \{n_k\}_{k=1}^\infty ,
   \end{cases}
$$

Since  $B_{n,2}=\{(n+1)^2 -1 + N_1 N_2 \cdots N_n\}\ \text{or} \ \emptyset$, we have
\[
\sum_{n=1}^{\infty} \frac{\# B_{n,2}}{\# B_n} < \sum_{n=1}^{\infty} \frac{1}{(n+1)^2}<  \infty,
\]
and it implies that  $\{( N_n,B_n)\}_{n=1}^\infty $ satisfies   RBC.
Note that the subsequence of $\set{B_{n_k}}_{k=1}^\infty$ consists of general consecutive sets, by Theorem \ref{infinitely many admissible pairs}, the infinite convolution $\mu$ exists and is a spectral measure.

On the other hand,  the probability measure $\mu$ is not compactly supported  since $$ \sum_{k=1}^{\f} \frac{\max\set{b: b \in B_{n_k}}}{N_1 N_2 \cdots N_{n_k}} =\sum_{k=1}^{\f} \left( \frac{(k+1)^2 -1}{N_1 N_2 \cdots N_{n_k}} + 1 \right) =\f, $$
 and  $\mu([0,n])<1$  for all $n\in \N$.

Alternatively, $\{(N_{n_k},B_{n_k})\}_{k=1}^\infty$ satisfies   PCC, and by Theorem \ref{gap-condition} ,   $\mu$ is a spectral measure.
\end{example}



\section*{Acknowledgements}

The authors are grateful to the referee for his helpful comments.
The authors also wish to thank Prof. Lixiang An and Prof Xingang He for their helpful comments.

\end{document}